\documentclass[11pt]{amsart} 
\usepackage{amsmath,amsthm}
\usepackage{amssymb}
\usepackage{amscd}
\usepackage{graphicx}
\usepackage{color}

\newtheorem{theorem}{Theorem}[section]
\newtheorem{lemma}[theorem]{Lemma}

\newtheorem{question}[theorem]{Question}

\theoremstyle{definition}

\newtheorem{remark}[theorem]{Remark}

\newcommand{\lk}{\ell k}

\newcommand{\frL}{\ensuremath{\mathfrak{L}}}

\newcommand{\tb}{\ensuremath{{\mbox{\tt tb}}}}
\newcommand{\rot}{\ensuremath{{\mbox{\tt rot}}}}

\begin{document}

\title{On Overtwisted Contact Surgeries} 

\author{S\.{i}nem Onaran}

\address{Department of Mathematics, Hacettepe University, 06800 Beytepe-Ankara, Turkey}
\email{sonaran@hacettepe.edu.tr}

\begin{abstract}
In this note, we obtain a new result concluding when contact $(+1/n)$--surgery is overtwisted. We give a counterexample to a conjecture by James Conway on overtwistedness of manifolds obtained by contact surgery. We list some problems related to the contact surgery.
\end{abstract}
\maketitle
 \setcounter{section}{0}

\section{Introduction}

Contact surgeries have been an essential tool for a long time to study contact $3$--manifolds. This paper is concerned with the behaviour of contact structures under contact surgeries, in particular under contact $(+1/n)$--surgeries. Such surgeries along stabilized Legendrian knots are studied in \cite{Ozb05}, in particular the $(+1)$-contact surgery case is studied in \cite{LS06} and later in \cite{Et1}.

Let $L$ be a nullhomologous oriented Legendrian knot in a contact $3$-manifold $(M, \xi)$. Throughout the paper let $\tb(L)$ denote its Thurston-Bennequin invariant, $\rot(L)$ its rotation number and $\chi(L)$ the Euler characteristic of its Seifert surface.

\begin{theorem} Let $L$ be a nullhomologous oriented Legendrian knot in a tight contact $3$-manifold. If $\tb(L) < 0$ and $\rot(L) > - \chi(L)$, then for any positive integer $n$, contact $(+1/n)$--surgery along $L$ is overtwisted.
\label{thm1}
\end{theorem}

In paticular, If $\tb(L) < 0$ and $\rot(L) > - \chi(L)$, then contact $(+1)$--surgery along $L$ is overtwisted. Note that $(+1)$--surgery in a tight contact $3$-manifold is not necessarily overtwisted. For example, a single contact $(+1)$--surgery in the tight $3$-sphere along the $\tb=-1$ Legendrian unknot yields the tight and Stein fillable contact structure on $S^1 \times S^2$, \cite{DGS}. 

We want to remark that John Etnyre informed us that in an unpublished work he and Vela-Vick showed that under the hypothesis of Theorem ~ \ref{thm1} the Legendrian knot destabilizes.

In overtwisted contact $3$-manifolds we have:
\begin{theorem} Let $L$ be a nullhomologous Legendrian knot in an overtwisted contact $3$-manifold. For any positive integer $n$, contact $(+1/n)$--surgery along $L$ is always overtwisted.
\label{thm2} 
\end{theorem}

Conjecture 6.13 from \cite{JC} states that if $L$ is a nullhomologous Legendrian knot with $\tb(L) \leq -2$, then contact $(+n)$--surgery on $L$ is overtwisted, for any positive integer $n < |\tb(L)|$. We give a negative answer to this conjecture by constructing a knot with $\tb = -3$ where $(+n)$--surgery along the knot is always tight for any integer $n \geq 2$. 

\begin{theorem} There exists nullhomologous Legendrian knots $L$ in a tight contact $3$--manifold with $\tb(L) \leq -2$ where contact $(+n)$--surgery on $L$ is tight, for any positive integer $n < |\tb(L)|$.
\label{thmcountereg}
\end{theorem}

Before giving the proof of theorems, we list some open problems for the reader here.

Recently, one of the fundamental problems of contact geometry has been solved by Wand, \cite{Wand}. He showed that Legendrian surgery (i.e. contact $(-1)$--surgery) preserves tightness but we are still left with a fundamental question:

\begin{question} \cite{Et1} When does $(+1)$--surgery preserve tightness?
\label{Q1.4}
\end{question}

The next question rephrases Question~\ref{Q1.4} in terms of the surgery dual knot.

\begin{question} When does $(-1)$--surgery in  overtwisted contact $3$--manifolds result in tight contact $3$--manifolds? \label{Q1.5}
\end{question}

In \cite{LOSS}, Lisca-Ozsv\'{a}th-Stipsicz-Szab\'{o} define an invariant $\frL(L)$ in the knot Floer homology group $HFK^{-}(−M,L)$ of an oriented, nullhomologous Legendrian knot $L$ in a contact $3$-manifold $(M, \xi)$. It would be interesting to know:  

\begin{question} Let $L$ be a nullhomologous oriented  Legendrian knot with $\frL(L) \neq 0$ in a tight contact $3$--manifold. Is $(+n)$--surgery on $L$ always tight?
\end{question}

A Legendrian knot in an overtwisted contact $3$--manifold with a tight complement is called \textit{nonloose} or \textit{exceptional}. 

\begin{question} Let $L$ be a nullhomologous oriented Legendrian knot with $\frL(L) \neq 0$ in an overtwisted contact $3$--manifold. Is $(-n)$--surgery on $L$ always tight?
\end{question}

We assume that the reader is familiar with the elements of contact topology and we refer the reader to \cite{Gei}, \cite{Et}, \cite{Et2} for all the necessary background; for fundamentals of contact structures and Legendrian knots. 
\section{Proof of Theorems}
\begin{proof}[Proof of Theorem ~ \ref{thm1}] Let $L$ be a nullhomologous oriented Legendrian knot with $\tb(L) := \tb < 0$, $\rot(L):=\rot  > - \chi(L)$. By an algorithm turning a rational contact surgery into a sequence of contact $\pm(1)$-surgeries in \cite{DGS}, contact $(+1/n)$--surgery along $L$ is equivalent to contact $(+1)$--surgeries along $n$ successive push-offs $L_1$, $L_2$,$\ldots$ , $L_n$ of $L$. Let $L'$ be $(n+1)$st push-off of $L$. Denote the surgery dual knot, the image of $L'$, in the surgered manifold as $L^*$. We will show that the surgery dual knot $L^*$ violates a generalization of Bennequin's inequality, Theorem 2.1 of \cite{BE}, which holds only for knots in tight contact $3$-manifolds. To do so, we compute the rational Thurston-Bennequin invariant $\tb_\mathbb{Q}(L^*)$ and the rational rotation number $\rot_\mathbb{Q}(L^*)$ of $L^*$.

Since $(+1/n)$--surgery on $L$ is a topological $(n\tb+1)/n$--surgery on $L$, the homological order $r$ of $L^*$ is $|n\tb+1|$. Note that if $\Sigma$ is a rational minimal genus Seifert surface of $L^*$, then topologically it is the image of a minimal genus Seifert surface of $L$ and hence $\chi(\Sigma) = \chi(L)$. Note also that $\lk(L',L) = \lk(L', L_i) = \tb$ for $i = 1, \ldots , n$.

Following Lemma $6.4$ of \cite{JC} which extends Lemma $2$ of \cite{GO13}, Lemma $6.6$ of \cite{LOSS}, cf. \cite{DM}, to more general contact 3-manifolds, the linking matrix $\mathbf{M}$ is the $(n \times n)$ matrix

\begin{displaymath}
\mathbf{M} =
\left( \begin{array}{ccccc}
\tb+1 & \tb & \tb & \ldots & \tb \\
\tb & \tb+1 & \tb & \ldots & \tb \\
\vdots & \tb & \tb+1 & \ldots & \vdots \\
\tb & \vdots & \vdots & \ddots & \tb \\
\tb & \tb & \ldots & \tb & \tb+1
\end{array} \right),
\end{displaymath} 
and $\mbox{\tt det}(\mathbf{M}) = n\tb + 1$. The extended matrix $\mathbf{M_0}$ is the $(n+1) \times (n+1)$ matrix $
\mathbf{M_0} =
\left( \begin{array}{cc}
0 & \tb \\
\tb & \mathbf{M}
\end{array} \right)$, and $\mbox{\tt det}(\mathbf{M_0}) = -n\tb^2$. Then we may compute
\begin{center}
$\displaystyle \tb_\mathbb{Q}(L^*) = \tb(L') + \frac{\mbox{\tt det}\mathbf{M_0}}{\mbox{\tt det}\mathbf{M}} = \tb + \frac{-n\tb^2}{n\tb+1} = \frac{\tb}{n\tb+1}$, and 

\begin{equation} \displaystyle
\begin{array}{ccl}
 \rot_\mathbb{Q}(L^*) &=&  \rot(L') - \Big \langle  \left( \begin{array}{cccc}
\rot(L_1) \\
\rot(L_2) \\
\vdots \\
\rot(L_n)

\end{array} \right), \mathbf{M}^{-1} \left( \begin{array}{cccc}
\lk(L', L_1) \\
\lk(L', L_2) \\
\vdots \\
\lk(L', L_n)

\end{array} \right)\Big \rangle \nonumber \\ &=& \bigskip \rot - \Big \langle \left( \begin{array}{cccc}
\rot \\
\rot \\
\vdots \\
\rot

\end{array} \right), \mathbf{M}^{-1} \left( \begin{array}{cccc}
\tb \\
\tb \\
\vdots \\
\tb

\end{array} \right)\Big \rangle \\ &=& \bigskip \rot - \Big \langle \left( \begin{array}{cccc}
\rot \\
\rot \\
\vdots \\
\rot

\end{array} \right), \left( \begin{array}{cccc}
\tb/(n\tb + 1) \\
\tb/(n\tb + 1) \\
\vdots \\
\tb/(n\tb+1)

\end{array} \right)\Big \rangle  \\ &=& 
\displaystyle \rot - \frac{n\rot\,\tb}{n\tb + 1} = \frac{\rot}{n\tb + 1}.
\end{array}
\end{equation}
\end{center}

Since $L$ is a Legendrian knot in tight contact $3$-manifold, Bennequin's inequality $\tb(L) + |\rot(L)| \leq -\chi(L)$ holds for $L$.

If $\tb(L) = \tb < 0$ and $\rot(L)=\rot > -\chi(L) \geq 0$, then $\tb(L) + |\rot(L)| = \tb(L) + \rot(L) \leq -\chi(L)$ or $-\tb(L) \geq \chi(L) + \rot(L)$.

We compute 
\begin{equation}
\begin{array}{ccl}
\displaystyle \tb_\mathbb{Q}(L^*) + |\rot_\mathbb{Q}(L^*)| \smallskip &=& \displaystyle  \frac{\tb}{n\tb+1} + | \frac{\rot}{n\tb+1} | \nonumber \bigskip \\ &=& \displaystyle  \frac{-\tb}{|n\tb+1|} +  \frac{\rot}{|n\tb+1|} \bigskip \\  &\geq & \displaystyle \frac{\chi(L) +2\rot }{|n\tb+1|} \geq  \frac{-\chi(L)}{|n\tb+1|}.  
\end{array}
\end{equation}
since $\rot(L)=\rot > -\chi(L)$. So the surgery dual knot $L^*$ violates a generalization of Bennequin's inequality. We conclude that $L^*$ is a knot in an overtwisted contact $3$-manifold.
\end{proof}

\begin{proof}[Proof of Theorem ~ \ref{thm2}] Let $L$ be a nullhomologous Legendrian knot in an overtwisted contact $3$-manifold. Likewise, contact $(+1/n)$--surgery along $L$ is equivalent to contact $(+1)$--surgeries along $n$ successive push-offs of $L$ as before. Contact $(+1)$--surgery along a Legendrian knot in an overtwisted contact $3$-manifold is always overtwisted. If it were tight, then contact $(-1)$--surgery along the surgery dual knot in the resulting tight contact $3$-manifold would give back the overtwisted contact $3$-manifold and that contradicts \cite{Wand}. Then contact $(+1)$--surgeries along the push offs of $L$ are overtwisted and hence contact $(+1/n)$--surgery along $L$ is overtwisted.
\end{proof}

\begin{lemma} Let $L$ be a Legendrian knot in a contact $3$--manifold $(M, \xi)$. If contact $(+1)$--surgery along $L$ is tight, then for relatively prime integers $p, q > 0$ and $q - p< 0$, contact $(+p/q)$--surgery along $L$ is tight.  
\label{ntight}
\end{lemma}

In particular, if contact $(+1)$--surgery along $L$ is tight, then contact $(+n)$--surgery along $L$ is tight for any integer $n \geq 2$.

\begin{proof} By Ding-Geiges-Stipsicz \cite{DGS}, contact $(+p/q)$--surgery along $L$ is equivalent to contact $(+1)$--surgery along $L$ and a contact $(-\frac{p}{p-q})$--surgery along its push off, say $L'$. By \cite{DGS} again, contact  $(-\frac{p}{p-q})$--surgery along $L'$ is equivalent to a sequence of $(-1)$--surgeries along push off's of $L'$ with some additional zigzags. If contact $(+1)$--surgery along $L$ is tight, then by \cite{Wand} remaining $(-1)$--surgeries result in a tight contact $3$-manifold.
\end{proof}

\begin{proof}[Proof of Theorem ~ \ref{thmcountereg}] Let $L$ be a Legendrian knot in Figure \ref{counterexample}. By \cite{DGS}, a single contact $(+1)$--surgery along a standard Legendrian unknot produces the unique tight and Stein fillable contact structure on $S^1 \times S^2$, and the further two contact $(-1)$--surgeries in Figure \ref{counterexample} then produce a Stein fillable and hence tight contact structure. Thus, contact $(+1)$--surgery along $L$ is tight. Then, by Lemma ~ \ref{ntight} contact $(+n)$--surgery along $L$ is tight for any $n \geq 2$.

\begin{figure}[h]
\centering
\includegraphics[width=2in]{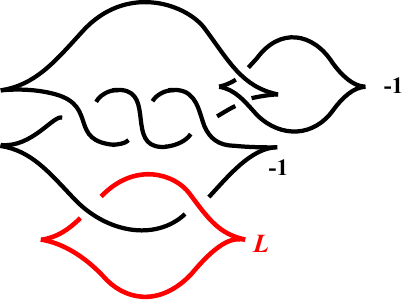}
\caption{Legendrian knot $L$ has $\tb=-3$ in the surgered manifold.}
\label{counterexample}
\end{figure}
 
Now, we check that $\tb(L)=-3$. For this we may use the $\tb$ formula from the proof of Theorem ~\ref{thm1}. Consider the linking matrix $\mathbf{M}$ and the extended linking matrix

\begin{displaymath}
\mathbf{M} =
\left( \begin{array}{cc}
\phantom{-}0 & -1 \\
-1 & -2
\end{array} \right), \; \;
\mathbf{M_0} =
\left( \begin{array}{ccc}
\phantom{-}0 & -1 & \phantom{-}0 \\
-1 & \phantom{-}0 & -1 \\
\phantom{-}0 & -1 & -2
\end{array} \right).
\end{displaymath}  

The Thurston-Bennequin invariant of $L$ in the unsurgered manifold is denoted by $\tb_0$ and here $\tb_0 = -1$. Then in the surgered manifold one has

\begin{displaymath} \tb(L) = \tb_0 + \frac{\mbox{\tt det}\mathbf{M_0}}{\mbox{\tt det}\mathbf{M}} = -1 + \frac{2}{-1} = -3. \qedhere
\end{displaymath} 
\end{proof}

\begin{remark} One can alternatively replace the knot $L$ in Figure~\ref{counterexample} by any knot having $\tb \leq -1$ in $(S^3, \xi_{std})$ where contact $(+1)$--surgery is tight. The same proof applies and may give new counterexample to Conway's conjecture. The knot $L$ in Figure~\ref{counterexample} is the simplest choice.
\end{remark}

\textbf{Acknowledgement} The author is partially supported by Turkish Academy of Sciences and T\"UB\.{I}TAK grant \#115F519.


\begin{thebibliography}{100}

\bibitem{BE} K. L. Baker, J. Etnyre, \emph{Rational linking and contact geometry},  Prog Math (2012), \textbf{296} 19--37.

\bibitem{JC} J. Conway, \emph{Transverse surgery on knots in contact $3$-Manifolds}, arXiv:1409.7077.

\bibitem{DG} F. Ding, H. Geiges, \emph{A Legendrian surgery presentation of contact $3$-manifolds}, Math Proc Cambridge (2004), \textbf{136}  583--598.

\bibitem{DGS} F. Ding, H. Geiges and A. I. Stipsicz, \emph{Surgery diagrams for contact $3$-manifolds}, Turk J Math (2004), \textbf{28} 41--74.

\bibitem{DM} S. Durst, M. Kegel, \emph{Computing rotation and self-linking numbers in contact surgery diagrams},  Acta Math Hung (2016), \textbf{150} no.2, 524--540.

\bibitem{Et} J. B. Etnyre, \emph{Introductory lectures on contact geometry}, In Topology and Geometry of Manifolds, Athens (2001), 81--107. Proceedings of Symposia in Pure Mathematics \textbf{71} Providence, RI: American Mathematical Society (2003).

\bibitem{Et2} J. B. Etnyre, \emph{Legendrian and transversal knots}, Handbook of knot theory, Elsevier B. V., Amsterdam (2005), 105--185.

\bibitem{Et1} J. B. Etnyre, \emph{On contact surgery}, P Am Math Soc (2008), \textbf{136} no. 9, 3355--3362.

\bibitem{Gei} H. Geiges, \emph{An introduction to contact topology}, Cambridge studies in advanced mathematics (2008), \textbf{109}.

\bibitem{GO13} H. Geiges, S. Onaran, \emph{Legendrian rational unknots in lens spaces}, J Symplect Geom (2015), \textbf{13}  17--50.

\bibitem{LS06} P. Lisca, A. I. Stipsicz, \emph{Notes on the contact Ozsv\'{a}th-Szab\'{o} invariants}, Pac J Math (2006), \textbf{228} no. 2, 277--295.

\bibitem{LOSS} P. Lisca, P. Ozsv\'{a}th, A. I. Stipsicz, and Z. Szab\'{o}, \emph{Heegaard Floer invariants of Legendrian knots in contact three-manifolds}, J Eur Math Soc (2009), \textbf{11}  no. 6, 1307--1363.

\bibitem{Ozb05} B. Ozbagci, \emph{A note on contact surgery diagrams}, Int J Math (2005), \textbf{16} no. 2, 87--99.

\bibitem{Wand} A. Wand, \emph{Tightness is preserved by Legendrian surgery}, Ann Math (2015), \textbf{182} no. 2, 723--738.

\end{thebibliography}
\end{document}